\documentclass[12pt]{amsart}
\usepackage[utf8]{inputenc}
\usepackage{amsfonts}
\usepackage{dsfont}
\usepackage{amssymb}
\usepackage{amsmath}
\usepackage{amsthm}
\usepackage{mathrsfs}
\usepackage{tikz-cd}
\usepackage{graphicx}
\usepackage{array}
\usepackage{mathtools}
\usepackage{hyperref}
\usepackage{cleveref}

\usepackage{enumitem}

\usepackage{blkarray}

\usepackage{multicol}

\usepackage{lipsum}

\usepackage{times}

\newcommand\ldot{\mathrel{\ooalign{<\cr
  \hidewidth\raise0.225ex\hbox{$\cdot\mkern0.5mu$}\cr}}}

\newcommand{\m}{\mathfrak{m}}

\newcommand{\N}{\mathds{N}}
\newcommand{\Z}{\mathds{Z}}
\newcommand{\C}{\mathds{C}}
\newcommand{\R}{\mathds{R}}

\newcommand{\Pp}{\mathds{P}}

\newcommand{\F}{\mathcal{F}}

\newcommand{\rk}{\text{rank }}
\newcommand{\Hilb}{\mbox{Hilb}}

\newcommand{\qand}{\quad\mbox{and}\quad}

\newcommand{~}{\thicksim}

\theoremstyle{definition}
\newtheorem{definition}{Definition}

\theoremstyle{definition}
\newtheorem{remark}[definition]{Remark}

\theoremstyle{definition}
\newtheorem{example}[definition]{Example}

\theoremstyle{plain}
\newtheorem{theorem}[definition]{Theorem}
\newtheorem{lemma}[definition]{Lemma}
\newtheorem{proposition}[definition]{Proposition}
\newtheorem{corollary}[definition]{Corollary}

\newtheorem{lemma*}{Lemma}

\newtheorem{question}[definition]{Question}

\begin{document}

\title[WLP and mixed multiplicities]{The Weak Lefschetz Property and Mixed Multiplicities of monomial ideals}
\author[T. Holleben]{Thiago Holleben}
\address[T. Holleben]
{Department of Mathematics \& Statistics,
Dalhousie University,
6316 Coburg Rd.,
PO BOX 15000,
Halifax, NS,
Canada B3H 4R2}
\email{hollebenthiago@dal.ca}
\date{}


\begin{abstract}
    Recently, H. Dao and R. Nair gave a combinatorial description of simplicial complexes $\Delta$ such that the squarefree reduction of the Stanley-Reisner ideal of $\Delta$ has the WLP in degree $1$ and characteristic zero. In this paper, we apply the connections between analytic spread of equigenerated monomial ideals, mixed multiplicities and birational monomial maps to give a sufficient and necessary condition for the squarefree reduction $A(\Delta)$ to satisfy the WLP in degree $i$ and characteristic zero in terms of mixed multiplicities of monomial ideals that contain combinatorial information of $\Delta$, we call them incidence ideals. As a consequence, we give an upper bound to the possible failures of the WLP of $A(\Delta)$ in degree $i$ in positive characteristics in terms of mixed multiplicities. Moreover, we extend Dao and Nair's criterion to arbitrary monomial ideals in positive odd characteristics.
\end{abstract}

\maketitle

\section{Introduction}

An Artinian graded $k$-algebra $A$, where $k$ is a field,  is said to satisfy the Weak Lefschetz Property (WLP) if the multiplication maps by a general linear form $L$, $\times L: A_i \to A_{i + 1}$ have full rank for all $i$. Moreover, $A$ is said to satisfy the Strong Lefschetz Property (SLP) if the multiplication maps by powers of a general linear form $L$, $\times L^d: A_i \to A_{i + d}$ have full rank for all $i$ and all $d$. Over the last few decades, Lefschetz properties and its connections to other areas of mathematics have been extensively studied. Areas where the connections to the WLP and SLP have been studied include Algebraic and Differential Geometry, Topology and Combinatorics (see for example: \cite{diffgeo}, \cite{mns}, \cite{coloredcook}, \cite{stanleyweyl}). 
 
One of the consequences of an algebra $A$ having the WLP is that its $h$-vector, which keeps track of the coefficients of the Hilbert series, is unimodal (see for example \cite{harima2003weak}). In view of this, studying the WLP of algebras where the $h$-vector is known to be equal to a sequence associated to a combinatorial object, is of interest. A particular class of algebras where the WLP has been recently studied is 

$$
    A(\Delta) = k[x_1, \dots, x_n]/((x_1^2, \dots, x_n^2) + I_\Delta)
$$
where $I_\Delta$ is a squarefree monomial ideal, the Stanley-Reisner ideal of $\Delta$. The $h$-vector of such algebras was shown to be equal to the $f$-vector of $\Delta$ in \cite{mns, daonair}.

In \cite{daonair} the authors completely described the simplicial complexes $\Delta$ such that $A(\Delta)$ that have the WLP in degree $1$, when the base field has characteristic $0$.

In this paper, we introduce incidence ideals of a simplicial complex $\Delta$, which can be seen as generalizations of facet ideals. We show that the analytic spread of such ideals determines the WLP of $A(\Delta)$ in characteristic zero.

\begin{theorem}[\Cref{wlpspread}]\label{wlpspreadintro}
    Let $\Delta$ be a simplicial complex and assume the base field is of characteristic $0$. Then
    
    \begin{enumerate}
        \item $A(\Delta)$ has the WLP in degree $i$ if and only if 
            $$
                \ell(I_\Delta(i)) = \min(f_{i - 1}, f_i),
            $$
            where $\ell(I_\Delta(i))$ is the analytic spread of the $i$-th incidence ideal of $\Delta$.
        \item If $A(\Delta)$ is a level algebra, then it has the SLP in degree $1$ if and only if 
            $$
                \ell(\F(\Delta(d))) = \min(f_0, f_d) \ \ \text{ for every $d$}
            $$
            where $\ell(\F(\Delta(d)))$ is the analytic spread of the facet ideal of the $d$-skeleton of $\Delta$.    
    \end{enumerate}
\end{theorem}

An open question in this area is to describe the positive characteristics where a given class of algebras defined by monomial ideals fails the WLP (see \cite[Question 7.3]{tour}). We answer this question for the algebras $A(\Delta)$ in terms of the mixed multiplicities of ideals contained in the incidence ideals of $\Delta$. In particular, we prove the following bound:   

\begin{corollary}[\Cref{failurewlpbound}]\label{failurewlpboundintro}
    Assume $A(\Delta)$ has the WLP in degree $i$ and characteristic zero. Assume also that $f_{i - 1} \leq f_{i}$. Then $A(\Delta)$ has the WLP in degree $i$ and characteristic $p$ for every $p$ such that
    $$
        p > e_{(0, f_{i - 1} - 1)}(\m_{\Delta, i} | I_\Delta(i)) \qand p \text{ does not divide } i + 1.
    $$
\end{corollary}

As a consequence of the two results above, we also generalize the criterion given in \cite[Theorem 3.3]{daonair} for algebras $A(\Delta)$ to have the WLP in degree $1$ and characteristic zero, to odd characteristics and arbitrary algebras $A$ defined by monomial ideals. 

\begin{theorem}[\Cref{failwlpdeg1arbitrary}]
    Let $I \subset R = k[x_1, \dots, x_n]$ be a monomial ideal such that $A = R/I$ is Artinian. Assume $\dim A_1 \leq \dim A_2$. Set

    $$
    E = \{x_i x_j | x_i x_j \neq 0 \text{ in $A$} \}, \text{ where $i$ may be equal to $j$}.
    $$ 
    then $A$ has the WLP in characteristic $0$ and degree $1$ if and only if it has the WLP in degree $1$ in every odd characteristic.

    Moreover, $A$ has the WLP in degree $1$ and characteristic $0$ if and only if the monomials in $E$ are the edges of a graph (possibly with loops) such that every connected component contains a (not necessarily induced) subgraph with either one loop or an odd cycle.
\end{theorem}

\section{Preliminaries}

Let $k$ be an infinite field and $R = k[x_1, \dots, x_n]$ the polynomial ring over $k$. Let $I$ be a homogeneous ideal such that $A = R/I$ is an Artinian graded algebra.

\begin{definition}
    Let $L$ be a general linear form. If the multiplication map $\times L: A_i \to A_{i + 1}$ has full rank for every $i$, we say $A$ has the \textbf{Weak Lefschetz Property (WLP)}.
    
    If moreover the multiplication maps by powers of $L$: $\times L^d: A_i \to A_{i + d}$ also have full rank for every $i$ and $d$, we say $A$ has the \textbf{Strong Lefschetz Property (SLP)}.
\end{definition}

\begin{proposition}[\cite{wlpmon}, Proposition 2.2]
    Let $I$ be a monomial ideal of $R$ such that $A = R/I$ is Artinian and $L = x_1 + \dots + x_n$. Then $A$ has the WLP (resp. SLP) if and only if the multiplication map by $L$ (resp. powers of $L$) has full rank.
\end{proposition}

We are particularly interested in the case where $I$ is the sum of a squarefree monomial ideal and the squares of the variables of $R$.

\begin{definition}
    A \textbf{simplicial complex} $\Delta$ with vertex set $V = \{1, \dots, n\}$ is a collection of subsets $\Delta$ of $V$ such that $\sigma \in \Delta$ and $\tau \subset \sigma$ implies $\tau \in \Delta$. Elements in $\Delta$ are called \textbf{faces} of $\Delta$, maximal faces are called \textbf{facets}. The \textbf{dimension} of $\Delta$ is $\dim \Delta = \max (|F| \colon F \text{ a facet of } \Delta) - 1$. If every facet of $\Delta$ has the same dimension, we say $\Delta$ is \textbf{pure}. 
    A $0$-dimensional face is called a \textbf{vertex} of $\Delta$, similarly, a $1$-dimensional face is called an \textbf{edge} of $\Delta$. The number of $i$-dimensional faces of $\Delta$ is denoted by $f_i(\Delta)$, or simply $f_i$ if the simplicial complex is clear from the context. By deleting every face of $\Delta$ with dimension higher than $i$, we get a new simplcial complex $\Delta(i)$ which is called the $i$\textbf{-skeleton } of $\Delta$.

    Given a simplicial complex $\Delta$, the ideal 
    $$
        I_\Delta = (x_{i_1}\dots x_{i_m} | \{i_1, \dots, i_m\} \not \in \Delta)
    $$
    is called the \textbf{Stanley-Reisner ideal} of $\Delta$. When we take the quotient of $R$ by $I_\Delta$ and the squares of the variables, we get an Artinian algebra:

    $$
        A(\Delta) := R/(I_\Delta + (x_1^2, \dots, x_n^2)).
    $$
\end{definition}

\begin{example}\label{example1}
    Let $\Delta$ be the complex with facets $\{\{a, b, c\}, \{a, c, d\}, \{b, c, d\}\}$.
    
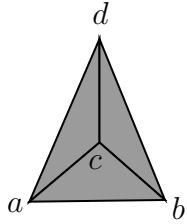
\begin{center}
    \tikzset{every picture/.style={line width=0.75pt}} 

    \begin{tikzpicture}[x=0.75pt,y=0.75pt,yscale=-1,xscale=1]

    \draw  [fill={rgb, 255:red, 155; green, 155; blue, 155 }  ,fill opacity=1 ] (117.08,36) -- (117.1,88.28) -- (81.76,118.14) -- cycle ;
    \draw  [fill={rgb, 255:red, 155; green, 155; blue, 155 }  ,fill opacity=1 ] (150.1,117.28) -- (117.1,88.28) -- (117.08,36) -- cycle ;
    \draw  [fill={rgb, 255:red, 155; green, 155; blue, 155 }  ,fill opacity=1 ] (150.1,117.28) -- (81.76,118.14) -- (117.1,88.28) -- cycle ;

    \draw (69,115) node [anchor=north west][inner sep=0.75pt]   [align=left] {$\displaystyle a$};
    \draw (112,16) node [anchor=north west][inner sep=0.75pt]   [align=left] {$\displaystyle d$};
    \draw (152.1,114.28) node [anchor=north west][inner sep=0.75pt]   [align=left] {$\displaystyle b$};
    \draw (110,93) node [anchor=north west][inner sep=0.75pt]   [align=left] {$\displaystyle c$};

    \end{tikzpicture}
\end{center}

Then $A(\Delta) = k[a,b,c,d]/(abd, a^2, b^2, c^2, d^2)$.
\end{example}

The algebra $A(\Delta)$ contains all the combinatorial information of $\Delta$:

\begin{proposition}[\cite{daonair}, Proposition 2.2]\label{allcombinfo}
    Let $\Delta$ be a simplicial complex. Then
    \begin{enumerate}
        \item The monomials in $A := A(\Delta)$ are in one-one correspondence with the faces of $\Delta$. For $i > 0$, $A_i$ is a $k$-vector space with a basis given by 
        $$
        \{x_F | F \text{ is an } (i - 1)\text{-face of } \Delta\}.
        $$
        \item The Hilbert Series of $A(\Delta)$, $\Hilb_{A(\Delta)}(t) = \sum_{i \geq 0} f_{i - 1}t^i$. That is, the $h$-vector of $A(\Delta)$ is equal to the $f$-vector of $\Delta$.
        \item $A(\Delta(i)) = \oplus_{j \leq i + 1} (A(\Delta))_j$.
    \end{enumerate}
\end{proposition}

\begin{definition}
    The ideal $(0 : A_{j \geq 1})$ is called the \textbf{socle} of $A$. The highest number $d$ such that $A_d \neq 0$ is called the \textbf{socle degree} and is denoted by $\text{socdeg(A)}$. 
    The algebra $A(\Delta)$ is said to be \textbf{level} of type $t$ if $(0: A_{j \geq 1}) = A_{\text{socdeg$(A)$}}$ and is generated by $t$ elements.
\end{definition}

\begin{proposition}[\cite{boijthesis}]
    The algebra $A(\Delta)$ is level if and only if $\Delta$ is a pure simplicial complex.
\end{proposition}

In \cite{daonair}, the authors completely described when $A(\Delta)$ has the WLP in degree $1$ and characteristic $0$ in terms of the $1$-skeleton of $\Delta$.

\begin{theorem}[\cite{daonair}, Theorem 3.3]\label{daonair}
    Let $\Delta$ be a simplicial complex, $R = k[x_1, \dots, x_n]$ where $k$ is a field of characteristic zero and 
    $$
        A(\Delta) = R/((x_1^2, \dots, x_n^2) + I_\Delta).
    $$
        
    Then 
    
    \begin{enumerate}
        \item If $f_1 \geq f_0$, then $A(\Delta)$ has the WLP in degree $1$ if and only if $\Delta(1)$ has no bipartite connected components.
        \item If $f_1 < f_0$, then $A(\Delta)$ has the WLP in degree $1$ if and only if each bipartite component of $\Delta(1)$ (if it exists) is a tree and each non-bipartite component satisfies the property that the number of edges in the component is equal to the number of vertices in the component. 
    \end{enumerate}
\end{theorem}

The proof of the theorem above uses the following well known result from graph theory:

\begin{proposition}\label{bip}
    Given a graph $G$ with $n$ vertices and $b_G$ bipartite connected components, the rank of the incidence matrix of $G$ over a field of characteristic zero is given by $n - b_G$.
\end{proposition}

\begin{remark}
    One particular case of \Cref{daonair} is when $\Delta$ has no isolated vertices and the same number of vertices and edges. Then $A(\Delta)$ has the WLP in degree $1$ if and only if every connected component of $\Delta(1)$ is a tree with a single odd cycle. As we will see, this version of the theorem has been proven from different perspectives to answer different questions.
\end{remark}

\section{Incidence matrices and analytic spread}\label{sectionincidence}

Let $I$ be an ideal of $R$. The algebra $R[It] := \oplus_{i \in \N} I^it^i \subset R[t]$ is called the \textbf{Rees algebra} of $I$.
Let $\m$ be the maximal homogeneous ideal of $R$. The quotient $\F(I) = R[It]/\m R[It]$ is called the \textbf{(special) fiber ring} of $I$. The dimension of the fiber ring of $I$ is called the \textbf{analytic spread} of $I$ and is denoted by $\ell(I)$.

When $I$ is a monomial ideal, the fiber ring $\F(I)$ is the monomial algebra generated by the generators of $I$. We now list some of the connections between the fiber ring of a monomial ideal $I$ and convex geometry that will be useful.

\begin{definition}
    Let $M$ be a set of lattice points in $\Z^n$. The \textbf{rank} of $M$ is the rank of the subgroup of $\Z^n$ generated by $M$ and is denoted by $\rk \Z M$.
\end{definition}

\begin{theorem}[\cite{rankanalyticspread}]\label{dimrank}
    Let $I$ be an equigenerated monomial ideal and $M$ the set of exponents of the generators of $I$. Then

    $$
        \ell(I) = \rk \Z M.
    $$
\end{theorem}

In order to use \Cref{dimrank} as a tool to detect the WLP, we first define the following polynomial ring:

\begin{definition}
    Let $I$ be an Artinian monomial ideal in a polynomial ring $R = k[x_1, \dots, x_n]$. We call the polynomial ring

    $$
        R_I = \C[t_m |m  \text{ is a monomial in $R$, } m \not \in I]
    $$

    the \textbf{incidence ring} of $I$. When $I = I_\Delta + (x_1^2, \dots, x_n^2)$ we write $R_\Delta$ for the incidence ring of $I$. By \Cref{allcombinfo}, there is a bijection between variables of $R_\Delta$ and the faces of $\Delta$, so we write $t_\tau$ for the variable that corresponds to the face $\tau$ under this bijection.

    We will denote by $R_{I, i}$ (or $R_{\Delta, i}$, in the squarefree case) the quotient 
    
    $$
        R_{I, i} := R_I/(t_m | \deg m \neq i).
    $$

    Moreover, we write $\m_{I, i}$ ($\m_{\Delta, i}$, in the squarefree case) for the maximal graded ideal of $R_{I, i}$ (resp. $R_{\Delta, i}$).  
    Note that $\dim R_{\Delta, i} = f_{i - 1}$. 
\end{definition}
Next we define the incidence matrices of a simplicial complex. As we will see, these are exactly the matrices that need to have full rank for the algebra $A(\Delta)$ to have the WLP.

\begin{definition}
    Let $\Delta$ be a simplicial complex and $F_i$ the set of all $i$-dimensional faces of $\Delta$. The matrix $M(\Delta, i)$ is the $f_{i}$ by $f_{i - 1}$ matrix such that the rows are labeled by $i$-faces, the columns are labeled by $i-1$-faces and:

    $$
    M(\Delta, i)_{jk} = 
    \begin{cases}
        1 \ \ \ \text{ if the $j$-th element of $F_{i}$ contains the $k$-th element of $F_{i - 1}$} \\
        0 \ \ \ \text{ otherwise}
    \end{cases}
    $$
    we call this matrix the \textbf{$i$-th incidence matrix} of $\Delta$.
\end{definition}

Taking the rows of a matrix to be the exponents of monomials, we can define the ideals associated to incidence matrices:

\begin{definition}\label{defincidenceideal}
    Let $\Delta$ be a simplicial complex and $M(\Delta, i)$ its $i$-th incidence matrix. The ideal 
    
    $$
        I_\Delta(i) = ( \prod_{\tau \subset \sigma_1, |\tau| = i} t_\tau, \dots, \prod_{\tau \subset \sigma_{f_i}, |\tau| = i} t_\tau) \subset R_{\Delta, i},
    $$
    where $\sigma_i$ are the $i$-dimensional faces of $\Delta$ is called the \textbf{$i$-th incidence ideal} of $\Delta$. 
\end{definition}

\begin{example}\label{example2}
    Let $\Delta$ be the simplicial complex from \Cref{example1}. Then we have:
    
    $$
    M(\Delta, 1) = \begin{blockarray}{ccccc}
        & a & b & c & d \\
      \begin{block}{c(cccc)}
        ab & 1 & 1 & 0 & 0 \\
        ac & 1 & 0 & 1 & 0 \\
        ad & 1 & 0 & 0 & 1 \\
        bc & 0 & 1 & 1 & 0 \\
        bd & 0 & 1 & 0 & 1 \\
        cd & 0 & 0 & 1 & 1 \\
      \end{block}
    \end{blockarray} \ \ \ \ \ \ \ \ \
    M(\Delta, 2) = \begin{blockarray}{ccccccc}
        & ab & ac & ad & bc & bd & cd \\
      \begin{block}{c(cccccc)}
        abc & 1 & 1 & 0 & 1 & 0 & 0 \\
        acd & 0 & 1 & 1 & 0 & 0 & 1 \\
        bcd & 0 & 0 & 0 & 1 & 1 & 1 \\
      \end{block}
    \end{blockarray}
    $$
    
    From the matrices we see that the first incidence ideal of $\Delta$ is 
    
    $$
    I_\Delta(1) = (\underbrace{t_a t_b}_{ab}, \underbrace{t_a t_c}_{ac}, \underbrace{t_a t_d}_{ad}, \underbrace{t_b t_c}_{bc}, \underbrace{t_b t_d}_{bd}, \underbrace{t_c t_d}_{cd}) \subset R_{\Delta, 1},
    $$
    that is, the edge ideal of the $1$-skeleton of $\Delta$. Moreover, we know the second incidence ideal of $\Delta$ is: 
    
    $$
    I_\Delta(2) = (\underbrace{t_{ab} t_{ac} t_{bc}}_{abc}, \underbrace{t_{ac} t_{ad} t_{cd}}_{acd}, \underbrace{t_{bc} t_{bd} t_{cd}}_{bcd}) \subset R_{\Delta, 2}.
    $$
\end{example}

\begin{proposition}\label{basicpropsincidence}
    Let $\Delta$ be a simplicial complex and $I_\Delta(i)$ the i-th incidence ideal of $\Delta$. Then

    \begin{enumerate}
        \item $I_\Delta(i)$ is an equigenerated squarefree monomial ideal generated in degree $i + 1$.
        \item For any two distinct monomials $m_1, m_2$ in a minimal generating set of $I_\Delta(i)$ we have $\deg \gcd(m_1, m_2) \leq 1$.
    \end{enumerate}
\end{proposition}

\begin{proof}
    \begin{enumerate}
        \item The number of $i-1$-dimensional faces inside an $i$-dimensional face is always equal to $i + 1$, therefore the generators of $I_\Delta(i)$ are monomial, squarefree and all have degree $i + 1$.
        \item Given two distinct $i - 1$-faces, their union has at least $i + 1$ elements, in particular, if the intersection of two $i$-faces $\sigma, \tau$ contains two $i-1$-faces, then $\sigma = \tau$. In terms of the generators of the ideal, this means if two generators of $I_\Delta(i)$ share two variables, then they must be equal.
    \end{enumerate}
\end{proof}

\begin{example}\label{example3}
    Let $J = (x_1 x_2 x_3, x_3 x_4 x_5, x_5 x_6 x_7, x_1 x_7 x_8) \subset k[x_1, \dots, x_n]$, $n \geq 8$. Is there a simplicial complex $\Delta$ such that $J = I_\Delta(2)$? We can try to answer this question as follows:

        \item 
    If such a $\Delta$ exists, then $x_1$ corresponds to the variable $t_{\{a, b\}}$, where $\{a, b\}$ is an edge of $\Delta$. This implies the generator $x_1 x_2 x_3$ corresponds to the monomial $t_{\{a, b\}}t_{\{a, c\}}t_{\{b, c\}}$, which in turn corresponds to the triangle $\{a, b, c\}$, and more specifically we can say $x_2$ corresponds to the variable $t_{\{a, c\}}$ and $x_3$ to the variable $t_{\{b, c\}}$. Similarly we conclude the generator $x_3 x_4 x_5$ corresponds to the monomial $t_{\{b, c\}}t_{\{b, d\}}t_{\{c, d\}}$ and thus $x_5$ corresponds to the variable $t_{\{\square, d\}}$ where either $\square = b$ or $\square = c$. Notice that if $a$ is a vertex of one of the edges of the triangle associated to the monomial $x_5 x_6 x_7$, then it would imply either $\{a, b\}$ or $\{a, c\}$ is an edge of this triangle, which then would imply either $x_2$ or $x_3$ divides $x_5 x_6 x_7$ (a contradiction). In particular, this triangle must contain a vertex $e$ that is not in the previous triangles and so the vertices of the triangle that corresponds to $x_5 x_6 x_7$ are $\{\square, d, e \}$. Applying the same argument to $x_7$ we conclude $x_1 x_7 x_8$ corresponds to the triangle $\{a, b, e\}$ and thus $x_7$ corresponds to the edge $\{b, e\}$ and in particular $\square = b$. Therefore a simplicial complex $\Delta$ such that $J$ is its second incidence ideal has facets: $\{\{a, b, c\}, \{b, c, d\}, \{b, d, e\}, \{a, b, e\}\}$.
\end{example}

\Cref{basicpropsincidence} gives us necessary conditions for an ideal to be the incidence ideal of some simplicial complex $\Delta$. One natural question that arises is whether these conditions are also sufficient. \Cref{examplefano} tells us that the conditions are not sufficient.
 
\begin{example}\label{examplefano}
    Let 
    $$
       I_F = (x_1 x_2 x_3, x_3 x_4 x_5, x_1 x_5 x_6, x_1 x_4 x_7, x_2 x_5 x_7, x_3 x_6 x_7, x_2 x_4 x_6) \subset k[x_1,\dots, x_n],
    $$
       with $n \geq 7$. Then:
           \item \begin{enumerate}
               \item Without loss of generality we may assume $x_1 x_2 x_3$ corresponds to the triangle $\{a, b, c\}$ where $x_1$ is the edge $\{a, b\}$, $x_2$ is the edge $\{a, c\}$ and $x_3$ is the edge $\{b, c\}$ 
               \item $x_3 x_4 x_5$ corresponds to the triangle $\{b, c, d\}$ where $x_4$ is the edge $\{\square, d\}$ and $x_5$ is the edge $\{\triangle, d\}$, where $b = \square$ and $c = \triangle$, or the opposite
               \item $x_1 x_5 x_6$ corresponds to the triangle $\{a, b, \triangle, d\}$, in particular $\triangle = b$
               \item $x_1 x_4 x_7$ corresponds to the triangle $\{a, b, \square, d\}$, in particular $\square = b$. This is a contradiction since it implies $x_4 = x_5$.
           \end{enumerate}
       
   In particular, $I_F$ satisfies the properties in \Cref{basicpropsincidence} but it is not the incidence ideal of any simplicial complex.
\end{example}

\begin{lemma}\label{matrixrepresentative}
    Let $\Delta$ be a simplicial complex, $A = A(\Delta)$ and $L = x_1 + \dots + x_n$. The matrix that represents the multiplication map $\times L: A_i \to A_{i + 1}$ is $M(\Delta, i)$, where the entries of $M(\Delta, i)$ are taken to be in $k$, the base field of $A$. 
\end{lemma}

\begin{proof}
    We know the monomials of $A_i$ are in bijection with the $i-1$-dimensional faces of $\Delta$. For a $i - 1$-dimensional face $\tau$,
    
    $$
    \times L(x_\tau) = L x_\tau = \sum_{\tau \cup \{j\} \in \Delta} x_\tau x_j
    $$
    thus the collumn corresponding to $\tau$ will have $1$ for each $i$-dimensional face $\sigma$ such that $\tau \subset \sigma$, which is the definition of the incidence matrix $M(\Delta, i)$.
\end{proof}

\begin{theorem}[\textbf{WLP from incidence matrices}]\label{wlpspread}
    Let $\Delta$ be a simplicial complex and assume the base field is of characteristic $0$. Then $A(\Delta)$ has the WLP in degree $i$ if and only if 

    $$
        \ell(I_\Delta(i)) = \min(f_{i - 1}, f_i).
    $$
\end{theorem}

\begin{proof}
    Let $M = \{\alpha_1, \dots, \alpha_{f_i}\}$ be the set of rows of $M(\Delta, i)$. The rank of the matrix $M(\Delta, i)$ is by definition the dimension of the space spanned by its rows. Since the field is of characteristic zero, this dimension is exactly $\rk \Z M$, therefore by \Cref{dimrank}

    $$
        \ell(I_\Delta(i)) = \rk \Z M = \rk M(\Delta, i).
    $$

    In particular, $M(\Delta, i)$ has full rank if and only if 
    $$
    \ell(I_\Delta(i)) = \min (f_{i - 1}, f_{i}).
    $$.
\end{proof}

\begin{remark}
    Connections between the bipartite property for graphs and algebraic invariants of edge ideals that are related to analytic spread can be found throughout the literature, see for example \cite{alilooee2018generalized} and \cite{bircomb}.
    The first incidence ideal of a simplicial complex $\Delta$ is the edge ideal of $\Delta(1)$, therefore, one can think of \Cref{daonair} as a generalization of these results in the language of the WLP.
\end{remark}

\begin{remark}
    All of the results and definitions above can be generalized to arbitrary monomial ideals $I$ such that $R/I$ is Artinian. The only difference is that when there is a nonzero nonsquarefree monomial in $R/I$, then one defines the incidence ideal to be generated by powers of the monomials in \Cref{defincidenceideal}, forcing the incidence ideal to be equigenerated.
\end{remark}

\begin{example}\label{example4}
    Let $\Delta$ be the simplicial complex in \Cref{example1} and assume the base field has characteristic zero. Using Macaulay2 \cite{M2} we can easily check that for $I_\Delta(1) = (t_a t_b, t_a t_c, t_a t_d, t_b t_c, t_b t_d, t_c t_d) \subset R_{\Delta, 1}$ and $I_\Delta(2) = (t_{ab} t_{ac} t_{bc}, t_{ac} t_{ad} t_{cd}, t_{bc} t_{bd} t_{cd}) \subset R_{\Delta, 2}$ we have:
    
    \begin{enumerate}
        \item $\ell(I_\Delta(1)) = 4 = f_0$
        \item $\ell(I_\Delta(2)) = 3 = f_2$
    \end{enumerate}

    By \Cref{wlpspread}, $\Delta$ has the WLP in all degrees.
\end{example}

Next we give a criterion for the Artinian algebra $A(\Delta)$ to satisfy the SLP in degree $1$ in terms of facet ideals.

\begin{definition}
    Let $\Delta$ be a simplicial complex on vertex set $\{1, \dots, n\}$ with facets $\{Q_1, \dots, Q_s\}$ and $k$ a field. The ideal 
    $$
        \F(\Delta) = (\prod_{j \in Q_1}t_j, \dots, \prod_{j \in Q_s}t_j) \subset R_{\Delta, 1}
    $$
    is called the \textbf{facet ideal} of $\Delta$. More generally, we write $\F(\Delta, i)$ for the ideal generated by the $i$-dimensional facets of $\Delta$. Note that when $\Delta$ is pure, $\F(\Delta, i)$ is the facet ideal of the $i$-skeleton $\Delta(i)$.
\end{definition}

\begin{theorem}[\textbf{SLP in degree one in terms of skeletons}]\label{slp1}
    Let $\Delta$ be a simplicial complex. Then $\Delta$ has the SLP in degree 1 and characteristic zero if and only if the following holds for all $d$:

    $$
        \ell(\F(\Delta(d), d)) = \min (f_{0}, f_d).
    $$

    In particular, if $\Delta$ is pure, then $A(\Delta)$ has the SLP in degree $1$ and characteristic zero if and only if the following holds for all $d$:

    $$
        \ell(\F(\Delta(d))) = \min(f_{0}, f_d).
    $$
\end{theorem}

\begin{proof}
    To prove that $A(\Delta)$ has the SLP in degree $1$, we need to check that the maps $\times L^d: A_1 \to A_{d + 1}$ have full rank for all $d$. Note that for every $x_j$ we have

    $$
        L^d x_j  = d! x_j\sum_{\{i_1, \dots, i_d\} \in \Delta} x_{i_1}\dots x_{i_d} = d! \sum_{\{j\} \cup \{i_1, \dots, i_d\} \in \Delta} x_j x_{i_1} \dots x_{i_d}.
    $$

    That is, the matrix that represents the linear transformation $\times L^d$ is:

    $$
    [\times L^d]_{jk} = \begin{cases}
        d! \ \ \ \text{ if the $j$-th $d$-face of $\Delta$ contains the $k$-th vertex of $\Delta$} \\
        0 \ \ \ \text{ otherwise}.
    \end{cases}
    $$

    Note that the rows of $[\times L^d]$ are in bijection with the facets of $\Delta(d)$ that have dimension $d$. Moreover, since the rows of $[\times L^d]$ are exactly the exponent vectors (after rescaling, which does not affect the rank in characteristic zero) of the monomials $\prod_{j \in \tau} x_j$, where $\tau$ is a facet of $\Delta(d)$ of dimension $d$, we conclude by \Cref{dimrank} that $[\times L^d]$ has full rank if and only if 

    $$
        \ell(\F(\Delta(d), d)) = \min(f_{0}, f_{d}).
    $$

    If $\Delta$ is pure, then $\Delta(d)$ is pure for every $d$, which means $\F(\Delta(d), d) = \F(\Delta(d))$, so the particular case follows.
\end{proof}

\begin{example}\label{example5}
    The simplicial complex $\Delta$ from \Cref{example1} is clearly a pure complex. Using Macaulay2 \cite{M2} we can easily compute 
    
    $$
        \ell(\F(\Delta)) = \ell(\F(\Delta(2))) = 3
    $$
    and 
    $$
        \ell(\F(\Delta(1))) = \ell(I_\Delta(1)) = 4,
    $$
    so by \Cref{slp1} $\Delta$ has the SLP in degree $1$ and characteristic zero.
\end{example}

\section{Birational Combinatorics}

Next we focus on rational maps from projective spaces defined by monomials. In \cite{bircomb}, Simis and Villarreal used the expression \textit{Birational Combinatorics} to describe the study of such maps. Most of their work was focused on maps $\Pp^n \dasharrow \Pp^n$. Here we state some of their results and mention their connections to the WLP. In later sections we will apply these results to the study of the WLP in positive characteristics.

\begin{definition}
    The incidence matrix of a graph $G = (V, E)$ with loops is the following:
    
    $$
        M(G)_{ij} = \begin{cases}
            1 \ \ \ \text{ if the edge $i$ is incident to the vertex $j$ and is not a loop} \\
            2 \ \ \ \text{ if the edge $i$ is a loop incident to $j$} \\
            0 \ \ \ \text{ otherwise}
        \end{cases}
    $$
\end{definition}

\begin{definition}
    Given a set of monomials $M$ in $k[x_1, \dots, x_n]$ the \textbf{log-matrix} of $M$ is the matrix $\log(M)$ such that its rows are the exponent vectors of the monomials in $M$. If the sum of every row of a matrix $N$ is equal to some fixed number $d$, we say $N$ is a \textbf{$d$-stochastic} matrix. In particular, if the monomials in $M$ have the same degree $d$, the log-matrix of $M$ is $d$-stochastic.

    If $M$ is a set of monomials of degree $2$, then $\log(M)$ is the incidence matrix of a graph $G$ (possibly with loops), the graph $G$ is called the \textbf{underlying graph} of $M$.
\end{definition}

\begin{definition}
    Given an ideal $I \subset k[x_1, \dots, x_n]$, there exists a surjective map $\psi : S(I) \to R[It]$, where $S(I)$ is the symmetric algebra of $I$. If $\psi$ is an isomorphism, the ideal $I$ is said to be of \textbf{linear type}. 
\end{definition}
 
\begin{definition}
    A finite set of monomials $M = \{m_1, \dots, m_s\} \subset k[x_1, \dots, x_n]$ of the same degree $d$ defines the following rational map:
    
    $$
        \varphi_M: \Pp^{n - 1} \dasharrow \Pp^{s - 1}, \ \ \ \varphi_M(x_1 : \dots : x_n) = (m_1 : \dots : m_s)
    $$

    If $n = s$ and $\varphi_M$ has an inverse rational map $\varphi'$ we say $\varphi_M$ is a \textbf{Cremona transformation} of $\Pp^{n - 1}$.
\end{definition}
 
One of the main applications given in \cite{bircomb} is the following description of birational monomial maps of degree $2$.

\begin{theorem}[\cite{bircomb}, Proposition 5.1]\label{simisvillarreal}
    Let $M \subset k[x_1, \dots, x_n]$ be a finite set of $n$ monomials of degree two with no non-trivial common factors such that the underlying graph $G$ is connected. Let $N$ denote the $n \times n$ incidence matrix of $G$. The following are equivalent:

    \begin{enumerate}
        \item det $N \neq 0$
        \item $\varphi_M$ is a Cremona transformation of $\Pp^{n - 1}$
        \item Either
            \begin{enumerate}
                \item $G$ has no loops and has a unique cycle of odd length
                \item $G$ is a tree with exactly one loop
            \end{enumerate}
        \item The ideal $(M) \subset k[x_1, \dots, x_n]$ is of linear type.
    \end{enumerate}
\end{theorem}

\begin{remark}
    \Cref{simisvillarreal} and \Cref{daonair} use very similar techniques to answer different questions from different fields, each with their natural restrictions. On one hand, \Cref{daonair} gives a criterion for the incidence matrix of any graph (without loops) to have full rank, regardless of whether the graph has more vertices or edges.
 
    On the other hand, the underlying graphs associated to rational maps $\Pp^{n  -1} \dasharrow \Pp^{n - 1}$ defined by degree $2$ monomials may have loops, but have the restriction that they must have the same number of vertices and edges.

    It is clear that in the case where the restrictions of the two areas intersect (connected graphs without loops and the same number of vertices and edges), then conditions $1.$ and $3.$ of \Cref{simisvillarreal} and \Cref{daonair} coincide.
\end{remark}

\begin{example}\label{example6}
    Let $G$ be the graph with vertex set $\{a,b,c,d\}$ and edges 
    
    $$
        \{\{a, c\}, \{a, d\}, \{b, c\}, \{c, d\}\}.
    $$ 
    
    Assume the base field has characteristic $0$. Then $G$ is a connected graph that contains only one odd cycle. By \Cref{simisvillarreal} the map $\Pp^{3} \dasharrow \Pp^3$ defined by $[a : b : c : d] \mapsto [ac : ad : bc : cd]$ is birational.

    Let $\Delta$ be the simplicial complex from \Cref{example1}. Note that $\Delta(1)$ contains $G$ as a subcomplex. This implies the incidence matrix of $G$ is a maximal square submatrix of $M(\Delta, 1)$. In particular, the determinant of the incidence matrix of $G$ is a maximal minor of $M(\Delta, 1)$. Again by \Cref{simisvillarreal} we know this determinant is not zero, and thus $M(\Delta, 1)$ has full rank. We have proven that $\Delta$ has the WLP in degree $1$ by finding a subcomplex $\Delta' \subset \Delta$ such that $f_0(\Delta) = f_0(\Delta') = f_1(\Delta')$ that has the WLP in degree $1$.
\end{example}

\begin{example}
    Consider the simplicial complex $\Gamma$ below:

    
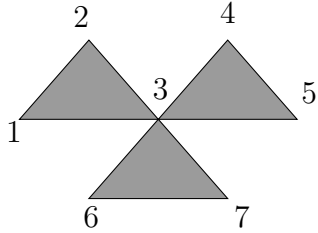
\begin{center}
    \begin{tikzpicture}[x=0.75pt,y=0.75pt,yscale=-1,xscale=1]
        
        \draw  [fill={rgb, 255:red, 155; green, 155; blue, 155 }  ,fill opacity=1 ] (135,122) -- (170,162) -- (100,162) -- cycle ;
        \draw  [fill={rgb, 255:red, 155; green, 155; blue, 155 }  ,fill opacity=1 ] (205,122) -- (240,162) -- (170,162) -- cycle ;
        \draw  [fill={rgb, 255:red, 155; green, 155; blue, 155 }  ,fill opacity=1 ] (170,162) -- (205,202) -- (135,202) -- cycle ;
        
        \draw (92,162) node [anchor=north west][inner sep=0.75pt]   [align=left] {$\displaystyle 1$};
        \draw (126,104) node [anchor=north west][inner sep=0.75pt]   [align=left] {$\displaystyle 2$};
        \draw (166,140) node [anchor=north west][inner sep=0.75pt]   [align=left] {$\displaystyle 3$};
        \draw (200,102) node [anchor=north west][inner sep=0.75pt]   [align=left] {$\displaystyle 4$};
        \draw (241,142) node [anchor=north west][inner sep=0.75pt]   [align=left] {$\displaystyle 5$};
        \draw (131,205) node [anchor=north west][inner sep=0.75pt]   [align=left] {$\displaystyle 6$};
        \draw (207,205) node [anchor=north west][inner sep=0.75pt]   [align=left] {$\displaystyle 7$};

        \end{tikzpicture}
    \end{center}

    The results in \cite{alisara} imply the facet ideal of $\Gamma$ is of linear type. Taking the vertices of $\Gamma$ to be the edges of another simplicial complex $\Delta$ (i.e, we want the facet ideal of $\Gamma$ to be the second incidence ideal of $\Delta$), we get the following:
    
    \begin{itemize}
        \item[] $1 \leftrightarrow \{a, b\}$, $2 \leftrightarrow \{a, c\}$, $3 \leftrightarrow \{b, c\}$, 
        $4 \leftrightarrow \{b, d\}$, $5 \leftrightarrow \{c, d\}$, $6 \leftrightarrow \{b, e\}$, $7 \leftrightarrow \{c, e\}$
    \end{itemize}

    In particular, the facet ideal of $\Gamma$ is the second incidence ideal of 
    $$
        \Delta = \{\{a, b, c\}, \{b, c, d\}, \{b, c, e\}\}.
    $$

    Moreover, we can verify that the algebra $A(\Delta)$ has the WLP in degree $2$ in every characteristic (at least one of the maximal minors of $M(\Delta, 2)$ is $\pm 1$).
\end{example}

The following result from \cite{bircomb} sheds some light into the failure of the WLP in positive characteristics. This will be our focus in future sections.

\begin{theorem}[\cite{detbir}, Proposition 2.1]\label{DPB}
    Let $M$ be a finite set of monomials of the same degree $d \geq 1$. Then $\varphi_M$ is a Cremona transformation if and only if 
    $$
        \det (\log(M)) = \pm d.
    $$
\end{theorem}

The next example shows how the failure of the WLP in positive characteristics is connected to a rational map not being birational.

\begin{example}\label{examplenonzerodet}
    Let $I = (x_3 x_4, x_4 x_6, x_1 x_7, x_4 x_7, x_5 x_7) \subset R = k[x_1, \dots, x_7]$ and $A = R/((x_1^2, \dots, x_7^2) + I)$. The multiplication map $\times L: A_2 \to A_3$, where $L = x_1 + \dots + x_7$ is a $16 \times 16$ square matrix with constant row sum $3$ and nonzero determinant, but the rational map defined by this matrix is not birational since the determinant is $-192$, instead of $\pm 3$. In particular, the first equivalence in \Cref{simisvillarreal} only works for monomial rational maps of degree $2$. Moreover, since $192 = 2^6 3$, we see that $A$ fails the WLP in characteristics $2$ and $3$.
\end{example}

\Cref{DPB} is a criterion for a map to be birational in terms of a determinant. It is useful to note that the determinant of a $d$-stochastic matrix is always divisible by $d$.

\Cref{simisvillarreal} together with \Cref{DPB} describe every possible value of the determinant of a $2$-stochastic square matrix with only $0, 1$ and $2$ as entries. \Cref{matrixrepresentative} describes the matrices that represent the multiplication maps that determine the WLP for Artinian algebras that are quotients of monomial ideals. The last result of this section is an application of \Cref{simisvillarreal} to the maximal minors of $2$-stochastic matrices that have more rows than columns. This result will then be used to describe the failure of the WLP in positive characteristics in degree $1$ in later sections.

\begin{corollary}\label{maincor}
    Let $M$ be a matrix with more rows than columns such that all of its entries are either $0, 1$ or $2$ and assume the sum of the entries in every row is $2$. Then every maximal minor of $M$ is either $0$ or its absolute value is a power of $2$.
\end{corollary}

\begin{proof}
    Since $M$ has more rows than columns, maximal minors of $M$ are determinants of submatrices of $M$ after deleting rows. Deleting rows does not change the row sum of other rows, so every maximal submatrix of $M$ also has constant row sum. In particular, every maximal submatrix can be seen as the log-matrix of a rational monomial map $\Pp^n \dasharrow \Pp^n$, where $n + 1$ is the number of columns of $M$. Note that if the gcd of all the monomials defining this rational map is not $1$, then either there are repeated monomials, in which case the value of the minor is $0$, or the underlying graph is a star with a loop, in which case the determinant of the log-matrix is $2$, so we may assume the gcd of all the monomials is $1$. We have two cases:

    \begin{enumerate}
        \item Assume first that the underlying graph of this rational map is connected. By \Cref{DPB} and \Cref{simisvillarreal}, either this maximal minor is $0$ (in case the map is not birational), or the map is birational and thus the maximal minor is $\pm 2$. 
        \item If the underlying graph is not connected, it is possible to rearrange the rows and columns so that the only nonzero entries of the matrix are square diagonal blocks (not necessarily of the same size) each block corresponds to a connected component of the underlying graph. We can then apply the argument above to each square diagonal block, and since the determinant of this submatrix is the product of the determinants of the diagonal blocks, it is either $0$ or a product of numbers which have absolute value $2$, so the result holds.
    \end{enumerate}

\end{proof}

\section{Mixed Multiplicities and the failure of the WLP in positive characteristics}

We now briefly introduce mixed multiplicities of ideals. For more details see \cite{mvmm} and \cite{trungpositivity}. 

Let $S = k[x_1, \dots, x_n]$ be a polynomial ring over a field $k$, $\m = (x_1, \dots, x_n)$ the maximal homogeneous ideal and $J_1, \dots, J_s$ an arbitrary sequence of proper ideals of $S$. The following is an $\N^{s + 1}$-graded standard $k$-algebra:

$$
    R(I | J_1, \dots, J_s) := \bigoplus_{(u_0, \dots, u_s) \in \N^{s + 1}} I^{u_0} J_1^{u_1} \dots J_s^{u_s} / I^{u_0 + 1} J_1^{u_1} \dots J_s^{u_s}
$$
denoting the length of the graded components of $R(I | J_1, \dots, J_s)$ as  
$$
    \ell(I^{u_0} J_1^{u_1} \dots J_s^{u_s} / I^{u_0 + 1} J_1^{u_1} \dots J_s^{u_s})
$$
we get a numerical function that for large values of $u_0, \dots, u_s$ is a polynomial of the form:

$$
    P(u) = \sum_{\alpha \in \N^{s + 1}, |\alpha | = n - 1} \frac{1}{\alpha!} e_\alpha u^\alpha + \ \{\text{terms of total degree $< r$ }\},
$$
where $u^\alpha = u_0^{\alpha_0}\dots u_s^{\alpha_s}$, $|\alpha| = \alpha_0 + \dots + \alpha_s$ and $\alpha! = \alpha_0! \dots \alpha_s!$. The numbers $e_\alpha$ are called the \textbf{mixed multiplicities of $\m, J_1, \dots, J_s$} and will be denoted as $e_\alpha(\m | J_1, \dots, J_s) := e_\alpha$.

Given a sequence of rational polytopes $Q_1, \dots, Q_r$ in $\R^n$ with $\dim (Q_1 + \dots + Q_r) \leq r$ we define the \textbf{mixed volume} of $Q_1, \dots, Q_r$ as

$$
    MV_r(Q_1, \dots, Q_r) := \sum_{h = 1}^r \sum_{1 \leq i_1 \leq \dots \leq i_h \leq r} (-1)^{r - h} V_r(Q_{i_1} + \dots Q_{i_h}),
$$
where $V_r(Q)$ denotes the $r$-dimensional Euclidean volume.

In \cite{mvmm}, Trung and Verma showed that mixed multiplicities of equigenerated monomial ideals can be understood as mixed volumes of sequences of rational polytopes.

\begin{theorem}[\cite{mvmm}, Theorem 2.4]\label{mvmm}
    Let $J_1, \dots, J_s$ be equigenerated monomial ideals of $k[x_0, \dots, x_n]$ of degree $d_i$ and $Q_i$ denote the convex hull of exponents of the dehomogenized generators of $J_i$ in $k[x_1, \dots, x_n]$ and $Q_0$ the standard simplex in $\R^{n}$.
    Denote by $(i_0 Q_0, \dots, i_s Q_s)$ the sequence 
    
    $$
        (Q_0, \dots, Q_0, Q_1, \dots, Q_1, \dots, Q_s, \dots, Q_s)
    $$
    where each polytope $Q_j$ appears $i_j$ times and $i_0 + \dots + i_s = n$, then

    $$
        MV_n(i_0 Q_0, \dots, i_s Q_s) = e_{(i_0, \dots, i_s)}(\m | J_1, \dots, J_s).
    $$
\end{theorem}

In the particular case of \Cref{mvmm} where the index of the mixed multiplicity $e_\alpha(\m | I_1, \dots, I_s)$ is a vector with only one nonzero entry. Then the mixed volume computation reduces to a volume computation.

\begin{lemma}\label{mixedvoltovol}
    Let $Q$ be a polytope in $\R^n$. Then
    
    $$
    MV_n(\underbrace{Q, \dots, Q}_{n \text{ times}}) = n! V_n(Q)
    $$
\end{lemma}
 
If we further assume that $Q$ is a polytope in $\R^n$ with exactly $n$ vertices, the volume computation reduces to the computation of a determinant:

\begin{lemma}\label{detprojection}
    Let $P \subset \R^{n + 1}$ be the convex hull of the points 
    $$
        x_0 = (x_{0,0}, \dots, x_{0, n}), \dots, x_n = (x_{n, 0}, \dots, x_{n, n}).
    $$
    and assume $|x_0| = \dots = |x_n| = d$. Then 

    $$
        n!V_n(Q) = \det 
        \begin{pmatrix}
            1 & \dots &  x_{0, n} \\
            \vdots & \ddots & \vdots \\
            1 & \dots & x_{n, n}    
        \end{pmatrix} = \frac{1}{d} \det 
        \begin{pmatrix}
         x_{0, 0} & \dots &  x_{0, n} \\
         \vdots & \ddots & \vdots \\
         x_{n, 0} & \dots & x_{n, n}    
        \end{pmatrix}
    $$
    where $Q$ is the projection of $P$ to the hyperplane $z_0 = 0$.
\end{lemma}

\begin{remark}
    Note that given a set of monomials $m_1, \dots, m_s \in k[x_1, \dots, x_s]$ that define a rational map $\varphi: \Pp^{s - 1} \dasharrow \Pp^{s - 1}$, the map $\varphi$ is birational if and only if the last mixed multiplicity of $\m = (x_1, \dots, x_s)$ and $I = (m_1, \dots, m_s)$ is $1$. 
\end{remark}

From these two observations we can determine for which characteristics the WLP fails for a simplicial complex that has the WLP in characteristic $0$.

\begin{theorem}[\textbf{WLP in positive characteristics and mixed multiplicities}]\label{failurewlp}
    Let $\Delta$ be a simplicial complex such that $A(\Delta)$ has the WLP in characteristic zero and degree $i$. Assume moreover that $f_{i - 1} \leq f_{i}$. Let $m_1, \dots, m_{f_{i}}$ be the generators of the $i$-th incidence ideal of $\Delta$, $I_\Delta(i)$. Denote by $I_S$ the ideal generated by a subset $S$ of the generators of $I_\Delta(i)$. Then $\Delta$ has the WLP in degree $i$ and characteristic $p$ if and only if $p$ does not divide

    $$
    (i + 1) \gcd\Big{(}e_{(0, f_{i - 1} - 1)}(\m_{\Delta, i} | I_S) \ : \ S \subset \{m_1, \dots, m_{f_{i}}\}, \ |S| = f_{i - 1}\Big{)}
    $$
\end{theorem}

\begin{proof}
    We know $A(\Delta)$ has the WLP in degree $i$ when the base field has characteristic zero, so we know at least one $f_{i - 1} \times f_{i - 1}$ minor of $M(\Delta, i)$ is nonzero.

    By \Cref{mvmm}, we know $e_{(0, f_{i - 1} - 1)}(\m_{\Delta, i} | I_S) = MV_{f_{i - 1} - 1}(Q, \dots, Q)$ where $Q$ is the simplex such that the vertices of $Q$ are the projections of the exponents of the monomials in $S$. \Cref{mixedvoltovol} implies 
    
    $$
    e_{(0, f_{i - 1} - 1)}(\m_{\Delta, i} | I_S) = (f_{i - 1} - 1)! V_{f_{i - 1} - 1}(Q).
    $$
    
    By \Cref{detprojection} and since the generators of $I_S$ all have the same degree $i + 1$, we conclude

    $$
        (i + 1) e_{(0, f_{i - 1} - 1)}(\m_{\Delta, i} | I_S) = \det \log(S).
    $$

    Since $\det \log(S)$ is a maximal minor of $M(\Delta, i)$, we see that $M(\Delta, i)$ has a maximal nonzero minor modulo $p$ if and only if there exists an $S$ such that $p$ does not divide $(i + 1) e_{(0, f_{i - 1} - 1)}(\m_{\Delta, i} | I_S)$, which is equivalent to saying $p$ does not divide the gcd of $(i + 1) e_{(0, f_{i - 1} - 1)}(\m_{\Delta, i} | I_S)$ where $S$ ranges over every subset of $\{m_1, \dots, m_{f_i}\}$ of size $f_{i - 1}$.
\end{proof}

\begin{example}\label{example7}
    In \Cref{example4} we verified that the $A(\Delta)$ has the WLP in characteristic zero and all degrees where $\Delta$ is the simplicial complex from \Cref{example1}. The maximal minors of $M(\Delta, 1)$ are in bijection with subgraphs of the $1$-skeleton of $\Delta$ that have $4$ edges. For each minor there are only two possible cases:

    \begin{enumerate}
        \item Either the minor is the determinant of the incidence matrix of a $4$-cycle, in which case it is zero
        \item or it is the determinant of the incidence matrix of a triangle with an extra edge. In this case the underlying graph satisfies the hypothesis of \Cref{simisvillarreal} and defines a birational map. Then by \Cref{DPB} we conclude that it is $\pm 2$. 
    \end{enumerate}
    
    In particular, the gcd of all the minors is $2$, so $A(\Delta)$ has the WLP in degree $1$ in every characteristic $\neq 2$.
\end{example}

From \Cref{failurewlp} we can give a lower bound so that if $A(\Delta)$ has the WLP in degree $i$ and characteristic zero, then $A(\Delta)$ has the WLP in degree $i$ and characteristic $p$ for every $p$ above the lower bound.
 
\begin{lemma}[\cite{huh}, Lemma 32]\label{ineqmm}
    Let $I, J \subset k[x_1, \dots, x_n]$ be monomial ideals generated in the same degree such that $J \subset I$. Then 

    $$
        e_{\alpha}(\m | J) \leq e_{\alpha}(\m | I)
    $$
    for every $\alpha \in \N^2$ such that $|\alpha| = n - 1$, where $\m = (x_1, \dots, x_n)$.
\end{lemma}

\begin{corollary}\label{failurewlpbound}
    Assume $A(\Delta)$ has the WLP in degree $i$ and characteristic zero. Assume also that $f_{i - 1} \leq f_{i}$. Then $A(\Delta)$ has the WLP in degree $i$ and characteristic $p$ for every $p$ such that
    $$
        p > e_{(0, f_{i - 1} - 1)}(\m_{\Delta, i} | I_\Delta(i)) \text{ and } p \text{ does not divide } i + 1.
    $$
\end{corollary}

\begin{proof}
    By \Cref{ineqmm}, since any ideal generated by a subset of the generators of $I_\Delta(i)$ is contained in $I_\Delta(i)$, we have:

    $$
        e_{(0, f_{i - 1} - 1)}(\m_{\Delta, i} | I_S) \leq e_{(0, f_{i - 1} - 1)}(\m_{\Delta, i} | I_\Delta(i)) < p
    $$
    for every subset $S$ of generators of $I_\Delta(i)$. The result holds by \Cref{failurewlpbound} since $p$ does not divide $(i + 1)e_{(0, f_{i - 1} - 1)}(\m_{\Delta, i} | I_S)$ for any subset $S$.
\end{proof}

So far we assumed $A(\Delta)$ has the WLP in characteristic zero and used mixed multiplicities to study the failure of the WLP in positive characteristics. Our next goal is to understand the failure of the WLP in characteristic zero in terms of mixed multiplicities.

One of the main questions in the theory of mixed multiplicities of ideals was on their positivity, that is, when can we guarantee $e_{\alpha}(\m | J) > 0$ for some ideal $J$ in a standard $\N$-graded $k$-algebra? In \cite{trungpositivity}, Trung answered the question by relating the positivity of mixed multiplicities of an ideal $J$ and the analytic spread of $J$. The result was recently generalized to sequences of ideals by F. Castillo, Y. Cid-Ruiz, B. Li, J. Montaño and N. Zhang in \cite{sequencepositivity}.
 
\begin{theorem}[\cite{sequencepositivity}, Theorem 4.4]\label{sequencepositivity}
    Let $J_1, \dots, J_s$ be an arbitrary sequence of ideals of $k[x_1, \dots, x_n]$, $\m$ the maximal homogeneous ideal and $\alpha \in \N^{s + 1}$ such that $|\alpha| = n - 1$. Then $e_{\alpha}(\m | J_1, \dots, J_s) > 0$ if and only if for every subset $\{i_1, \dots, i_m\} \subset \{1, \dots, s\}$ the following inequality holds:

    $$
    \alpha_{i_1} + \dots + \alpha_{i_m} \leq \ell(J_{i_1}\dots J_{i_m}) - 1.
    $$
\end{theorem}

When $s = 1$, there is the result by Trung in \cite{trungpositivity}:
 
\begin{theorem}[\cite{sequencepositivity}, Corollary 3.7]\label{positivityone}
    Let $J$ be an arbitrary ideal of $k[x_1, \dots, x_n]$ and $\m = (x_1, \dots, x_n)$. Then $e_{(n - 1 - i, i)}(\m | J) > 0$ if and only if $0 \leq i \leq \ell(J) - 1$. 
\end{theorem}

As a consequence, we have the following characterization of the WLP in characteristic zero in terms of mixed multiplicities.

\begin{theorem}[\textbf{WLP in characteristic zero and mixed multiplicities}]\label{wlpmm}
    The squarefree reduction $A(\Delta)$ of the Stanley-Reisner ideal of a simplicial complex $\Delta$ has the WLP in characteristic zero and degree $i$ if and only if one of the following two conditions hold:
    
    \begin{enumerate}
        \item $f_{i - 1} \leq f_{i}$ and $e_{(0, f_{i - 1} - 1)}(\m_{\Delta, i} | I_\Delta(i)) > 0$
        \item $f_i \leq f_{i - 1}$ and $e_{(f_{i - 1} - f_{i}, f_i - 1)}(\m_{\Delta, i} | I_\Delta(i)) > 0$
    \end{enumerate}
\end{theorem}

\begin{proof}

    By \Cref{sequencepositivity}, we know $e_{(a, b)}(\m_{\Delta, i} | I_\Delta(i)) > 0$ if and only if $b\leq \ell(I_\Delta(i)) - 1$, where $a + b = f_{i - 1} - 1$. By \Cref{wlpspread}, $A(\Delta)$ has the WLP in degree $i$ and charactersitic zero if and only if $\ell(I_\Delta(i)) = \min(f_{i - 1}, f_i)$. We have two cases:

    \begin{enumerate}
        \item If $f_{i - 1} \leq f_i$, then $e_{(0, f_{i - 1} - 1)}(\m_{\Delta, i} | I_\Delta(i)) > 0$ if and only if 
        $$
            f_{i - 1} - 1 \leq \ell(I_\Delta(i)) - 1 \leq f_{i - 1} - 1
        $$
        \item If $f_i \leq f_{i - 1}$, then $e_{(f_{i - 1} - f_{i}, f_i - 1)}(\m_{\Delta, i} | I_\Delta(i)) > 0$ if and only if 
        
        $$
            f_{i} - 1 \leq \ell(I_\Delta(i)) - 1 \leq f_{i} - 1.
        $$
    \end{enumerate}
\end{proof}

By applying the results from \cite{bircomb} we extend the characterization of the WLP in characteristic zero and degree $1$ of \cite{daonair} to positive characteristics. In order to extend their result, we need the following lemma.

\begin{lemma}\label{disconnectedgraphwlp}
    Let $G$ be a graph such that $|V(G)| \leq |E(G)|$. If $G$ has a connected component with less edges than vertices, the incidence matrix $M$ of $G$ does not have full rank.
\end{lemma}

\begin{proof}
    The incidence matrix of $G$ is a block diagonal matrix where the blocks are the incidence matrices of the connected components $G_1, \dots, G_s$ of $G$. The rank of $M$ is equal to the sum of the ranks of each block. We have the following inequalities:

    $$
        \rk M \leq \sum_{i = 0}^s \min(|V(G_i)|, |E(G_i)|) < \sum_{i = 0}^s |V(G_i)| = |V(G)|. 
    $$
\end{proof}

\begin{theorem}[\textbf{WLP in degree 1 for squarefree monomial ideals}]\label{failwlpdeg1}
    Let $\Delta$ be a simplicial complex such that $f_0 \leq f_1$. Then one of the two following statements holds:

    \begin{enumerate}
        \item $A(\Delta)$ fails the WLP in degree $1$ in every characteristic.
        \item $A(\Delta)$ has the WLP in degree $1$ in every characteristic except characteristic $2$.
    \end{enumerate}
\end{theorem}

\begin{proof}
    The matrix $M(\Delta, 1)$ that represents the multiplication map 
    $$
        \times L: A(\Delta)_1 \to A(\Delta)_2
    $$
    has constant row sum $2$. By \Cref{maincor}, the gcd of its maximal minors is either $0$, or a power of $2$. If it is $0$, then $M(\Delta, 1)$ does not have full rank and so $A(\Delta)$ fails the WLP in degree $1$ in every characteristic. If it is a power of $2$, then the multiplication map $\times L: A(\Delta)_1 \to A(\Delta)_2$ has full rank whenever the characteristic of the base field is not $2$. Moreover, the multiplication map does not have full rank when the characteristic of the base field is $2$, so the result holds.
\end{proof}

\section{The WLP in degree 1 of arbitrary monomial ideals}

\Cref{failwlpdeg1} tells us in which characteristics the WLP in degree $1$ of a squarefree Artinian reduction $A$ satisfying $\dim A_1 \leq \dim A_2$ can fail. In this section we prove that the same result holds for arbitrary Artinian algebras $R/I$ where $I$ is a monomial ideal.

\begin{definition}
    Let $I \subset k[x_1, \dots, x_n] = R$ be a monomial ideal such that $A = R/I$ is Artinian. The \textbf{underlying graph} $G_A$ of $A$ has vertex set $\{x_i | x_i \not \in I\}$ and $\{x_i, x_j\}$ (possibly with $i = j$) is an edge of $G_A$ if and only if $x_i x_j \not \in I$. 
\end{definition}

\begin{example}
    Let $R = k[x_1, x_2, x_3, x_4]$ and $I = (x_1 x_3, x_1 x_2 x_4, x_1^3, x _2^2, x_3^5, x_4^2)$. Then the underlying graph of $A = R/I$ is the following:

    \begin{center}
        \tikzset{every picture/.style={line width=0.75pt}} 
        \begin{tikzpicture}[x=0.75pt,y=0.75pt,yscale=-1,xscale=1]
    
            \draw   (149.28,101) -- (220.28,101) -- (220.28,170.27) -- (149.28,170.27) -- cycle ;
            \draw    (149.28,101) -- (220.28,101) ;
            \draw [shift={(220.28,101)}, rotate = 0] [color={rgb, 255:red, 0; green, 0; blue, 0 }  ][fill={rgb, 255:red, 0; green, 0; blue, 0 }  ][line width=0.75]      (0, 0) circle [x radius= 3.35, y radius= 3.35]   ;
            \draw [shift={(149.28,101)}, rotate = 0] [color={rgb, 255:red, 0; green, 0; blue, 0 }  ][fill={rgb, 255:red, 0; green, 0; blue, 0 }  ][line width=0.75]      (0, 0) circle [x radius= 3.35, y radius= 3.35]   ;
            \draw    (149.28,170.27) -- (220.28,170.27) ;
            \draw [shift={(220.28,170.27)}, rotate = 0] [color={rgb, 255:red, 0; green, 0; blue, 0 }  ][fill={rgb, 255:red, 0; green, 0; blue, 0 }  ][line width=0.75]      (0, 0) circle [x radius= 3.35, y radius= 3.35]   ;
            \draw [shift={(149.28,170.27)}, rotate = 0] [color={rgb, 255:red, 0; green, 0; blue, 0 }  ][fill={rgb, 255:red, 0; green, 0; blue, 0 }  ][line width=0.75]      (0, 0) circle [x radius= 3.35, y radius= 3.35]   ;
            \draw   (133.28,61.27) .. controls (163.28,47.27) and (170.28,87.73) .. (149.28,101) .. controls (128.28,114.27) and (103.28,75.27) .. (133.28,61.27) -- cycle ;
            \draw   (247.28,201.27) .. controls (266.28,172.27) and (233.28,156.27) .. (220.28,170.27) .. controls (207.28,184.27) and (228.28,230.27) .. (247.28,201.27) -- cycle ;
            \draw    (149.28,170.27) -- (220.28,101) ;
    
            \draw (130.18,80) node [anchor=north west][inner sep=0.75pt]   [align=left] {$\displaystyle x_{1}$};
            \draw (221.28,79) node [anchor=north west][inner sep=0.75pt]   [align=left] {$\displaystyle x_{2}$};
            \draw (222.28,173.27) node [anchor=north west][inner sep=0.75pt]   [align=left] {$\displaystyle x_{3}$};
            \draw (130.18,171) node [anchor=north west][inner sep=0.75pt]   [align=left] {$\displaystyle x_{4}$};
        \end{tikzpicture}
    \end{center}

    The incidence matrix of $G_A$ is:

    $$
    \begin{blockarray}{ccccc}
        & x_1 & x_2 & x_3 & x_4 \\
      \begin{block}{c(cccc)}
        x_1^2 & 2 & 0 & 0 & 0 \\
        x_3^2 & 0 & 0 & 2 & 0 \\
        x_1 x_2 & 1 & 1 & 0 & 0 \\
        x_1 x_4 & 1 & 0 & 0 & 1 \\
        x_2 x_3 & 0 & 1 & 1 & 0 \\
        x_2 x_4 & 0 & 1 & 0 & 1 \\
        x_3 x_4 & 0 & 0 & 1 & 1 \\
      \end{block}
    \end{blockarray}
    $$

    Moreover, the multiplication map $\times L: A_1 \to A_2$, where $L$ is the sum of the variables, is represented by the following matrix:

    $$
    \begin{blockarray}{ccccc}
        & x_1 & x_2 & x_3 & x_4 \\
      \begin{block}{c(cccc)}
        x_1^2 & 1 & 0 & 0 & 0 \\
        x_3^2 & 0 & 0 & 1 & 0 \\
        x_1 x_2 & 1 & 1 & 0 & 0 \\
        x_1 x_4 & 1 & 0 & 0 & 1 \\
        x_2 x_3 & 0 & 1 & 1 & 0 \\
        x_2 x_4 & 0 & 1 & 0 & 1 \\
        x_3 x_4 & 0 & 0 & 1 & 1 \\
      \end{block}
    \end{blockarray}
    $$
\end{example}

The example above leads us to the following result:

\begin{proposition}\label{multiply2}
    Let $I$ be a monomial ideal of $R = k[x_1, \dots, x_n]$ and $A = R/I$ an Artinian algebra. Assume that the characteristic of $k$ is not $2$. Then the multiplication map $\times L: A_1 \to A_2$ has the same rank as the incidence matrix of $G_A$. 
\end{proposition}

\begin{proof}
    The multiplication map $\times L : A_1 \to A_2$ is represented by the matrix:

    $$
    M_{ij} = \begin{cases}
        1 \ \ \ \text{ if the $i$-th monomial of $A_2$ is divisible by the $j$-th monomial of $A_1$} \\
        0 \ \ \ \text{ otherwise}
    \end{cases}
    $$

    In particular, if we multiply the rows associated to monomials that are pure powers of variables by $2$, we get the incidence matrix of $G_A$. Since the characteristic of $k$ is not $2$, multiplying a row of a matrix by a nonzero element does not change the rank, so the result follows.
\end{proof}

We can then generalize the main result of \cite{daonair} to arbitrary monomial ideals in any characteristic different from $2$:

\begin{theorem}[\textbf{Main Theorem}]\label{failwlpdeg1arbitrary}
    Let $R = k[x_1, \dots, x_n]$ and $I$ a monomial ideal. Assume $A = R/I$ is Artinian and $\dim A_1 \leq \dim A_2$. Then either $A$ has the WLP in degree $1$ in every characteristic $\neq 2$, or $A$ fails the WLP in degree $1$ in every characteristic $\neq 2$. 
    
    Moreover, $A$ has the WLP in characteristic zero if and only if every connected component of $G_A$ contains a subgraph that is either:
    
    \begin{enumerate}
        \item A tree with one loop
        \item A graph that has only one cycle of odd length
    \end{enumerate}
\end{theorem}

\begin{proof}
    Assume $G_A$ is not a star with a loop at the center vertex. By \Cref{multiply2}, the rank of the matrix that represents the multiplication map $\times L : A_1 \to A_2$ is equal to the rank of the incidence matrix of $G_A$ in every characteristic $\neq 2$. By \Cref{disconnectedgraphwlp}, we may assume every connected component of $G_A$ has more vertices than edges. Since the combinatorial conditions of the graph are the same as the ones in \Cref{simisvillarreal}, the result follows by \Cref{maincor} and noticing that a maximal minor is nonzero if and only if the associated maximal submatrix is the incidence matrix of a graph where every connected component defines a birational map.
    
    If $G_A$ is a star with a loop, then we claim the matrix that represents the multiplication map $\times L: A_1 \to A_2$ has determinant $1$. To compute this determinant, subtract every column from the column associated to the vertex that is the center of the star. The resulting matrix is the identity matrix, so the result follows. 
\end{proof}

\begin{example}
    Let $A = k[x_1, x_2, x_3, x_4, x_5]/(x_1 x_3, x_1 x_2 x_4, x_1^3, x _2^2, x_3^5, x_4^2)$. Then $A$ has the WLP in degree $1$ in characteristic zero and every odd characteristic by \Cref{failwlpdeg1arbitrary}. In characteristic $2$ we can check $A$ also has the WLP in degree $1$ since at least one of the maximal minors of the matrix that represents the multiplication map $\times L: A_1 \to A_2$ is $1$.
\end{example}

\begin{remark}[\textbf{Another perspective: Hyperplane arrangements and matroids}]
    Throughout this paper, we took the rows of the incidence matrices of a simplicial complex $\Delta$, to be exponents of monomials in the incidence ring of $\Delta$. For an arbitrary monomial ideal $I \subset R$ over a field of characteristic zero, we could also take the entries in each column of an incidence matrix of $I$ (i.e, a matrix that represent the multiplication map $\times L: A_i \to A_{i + 1}$, where $A = R/I$ for some $i$) to be coefficients of a linear form in the incidence ring of $I$. For example taking $\Delta$ to be the simplicial complex from \Cref{example1} we have:    

    $$
    \begin{blockarray}{ccccc}
        & a & b & c & d \\
      \begin{block}{c(cccc)}
        ab & 1 & 1 & 0 & 0 \\
        ac & 1 & 0 & 1 & 0 \\
        ad & 1 & 0 & 0 & 1 \\
        bc & 0 & 1 & 1 & 0 \\
        bd & 0 & 1 & 0 & 1 \\
        cd & 0 & 0 & 1 & 1 \\
      \end{block}
    \end{blockarray}
    $$

    defines the product of linear forms:

    $$
    h = (\underbrace{t_{ab} + t_{ac} + t_{ad}}_a)(\underbrace{t_{ab} + t_{bc} + t_{bd}}_b)(\underbrace{t_{ac} + t_{bc} + t_{cd}}_c)(\underbrace{t_{ad} + t_{bd} + t_{cd}}_d) \in R_I.
    $$

    Combining the results on positivity of mixed multiplicities in \cite{trungpositivity} and the results on mixed multiplicities of jacobian ideals of products of linear forms in \cite{huh}, one concludes that the rank of these incidence matrices is equal to the analytic spread of jacobian ideals, for example:

    $$
        \rk (M(\Delta, 1)) = 4 = \ell(J_h)
    $$

    where $J_h \subset R_I$ is the jacobian ideal of $h$. 
    
    From this perspective, one natural question to ask is which matroids can be represented by incidence matrices. As an example, let $\Delta$ be the simplex on $5$-vertices, so that $A(\Delta) = k[a,b,c,d,e]/(a^2, b^2, c^2, d^2, e^2)$. Then the matrix
    
    $$
    M(\Delta, 3) = 
    \begin{blockarray}{ccccccccccc}
        & abc & abd & abe & acd & ace & ade & bcd & bce & bde & cde \\
      \begin{block}{c(cccccccccc)}
        abcd & 1 & 1 & 0 & 1 & 0 & 0 & 1 & 0 & 0 & 0 \\
        abce & 1 & 0 & 1 & 0 & 1 & 0 & 0 & 1 & 0 & 0 \\
        abde & 0 & 1 & 1 & 0 & 0 & 1 & 0 & 0 & 1 & 0 \\
        acde & 0 & 0 & 0 & 1 & 1 & 1 & 0 & 0 & 0 & 1 \\
        bcde & 0 & 0 & 0 & 0 & 0 & 0 & 1 & 1 & 1 & 1 \\
      \end{block}
    \end{blockarray}
    $$
    which corresponds to the triangle-tetrahedra incidence in $\Delta$, is a representation of the matroid $R_{10}$.
\end{remark}

\section{Closing remarks}

In this paper, we introduced incidence ideals of a simplicial complex $\Delta$ in order to study Lefschetz properties of the algebra $A(\Delta)$. In \Cref{basicpropsincidence}, we showed some necessary conditions for a squarefree monomial ideal to be the incidence ideal of some simplicial complex. \Cref{examplefano} gives us an example that these conditions are not sufficient, a natural question is: 

\begin{question}
    What are necessary and sufficient conditions for a squarefree monomial ideal to be the incidence ideal of some simplicial complex?
\end{question}

For a birational map $\varphi$ satisfying the hypothesis of \Cref{simisvillarreal}, the ideal generated by each entry of $\varphi$ is of linear type. In \cite{edgelineartype}, Villarreal characterized edge ideals of linear type as edge ideals of graphs that are trees or have one unique cycle of odd length. In \cite{alisara}, the authors generalized this result to all squarefree monomial ideals by introducing the concept of an even walk in a simplicial complex. In view of \Cref{simisvillarreal}, \Cref{failurewlp} and the fact that the last mixed multiplicity of an ideal that defines a birational map is $1$, we ask the following question:

\begin{question}
    Can we find classes of simplicial complexes where the algebra $A(\Delta)$ has the WLP in degree $i$ in every characteristic that does not divide $i + 1$ by finding incidence ideals of linear type? 
\end{question}

\bigskip

\textbf{Acknowledgements:} I would like to thank my supervisor Sara Faridi for very useful discussions which improved the presentation of the paper. I also thank Jonathan Montaño for answering questions on mixed multiplicities.














\end{document}